\documentclass[12pt]{article}
\usepackage[top=2.5cm,bottom=2.5cm,left=3.2cm,right=3.2cm]{geometry}
\usepackage{amssymb}
\usepackage{color}
\usepackage{amsfonts}
\usepackage{latexsym}
\usepackage{graphics}
\usepackage{epsf, epsfig}
\usepackage{mathdots}
\usepackage{bm}
\usepackage{amssymb}
\usepackage{amsthm}
\usepackage{xcolor}
\usepackage{amsmath}

\newtheorem{defi}{Definition}
\newtheorem{lem}{Lemma}
\newtheorem{theo}{Theorem}
\newtheorem{pro}{Proposition}
\newtheorem{cor}{Corollary}
\newtheorem{prob}{Problem}
\newtheorem{cla}{Claim}
\newtheorem{rem}{Remark}
\def \CR { {\rm cr} }

\colorlet{darkred}{red!60!black}
\colorlet{darkblue}{blue!60!black}
\usepackage{hyperref}
\hypersetup{
    colorlinks=true,
    citecolor=darkblue,     
    linkcolor=darkred,       
    urlcolor=green    
}

\usepackage[mathlines]{lineno}

\usepackage{enumitem}
\begin{document}
\newcommand{\resection}[1]
{\section{#1}\setcounter{equation}{0}}

\renewcommand{\theequation}{\thesection.\arabic{equation}}

\renewcommand{\labelenumi}{\rm(\roman{enumi})}

\baselineskip 0.6 cm

\title {\bf Maximal 1-plane graphs with the maximum number of crossings\thanks {The work was supported by the National
Natural Science Foundation of China (Grant Nos. 12371346, 12271157)}}

\author{Zhangdong Ouyang\\
{\footnotesize Department of Mathematics, Hunan First Normal University , Changsha 410205, P.R.China} \\
{\footnotesize E-mail:  oymath@163.com }\\
Yuanqiu Huang\\
{\footnotesize  Department of Mathematics, Hunan Normal University, Changsha 410081, P.R.China} \\
{\footnotesize E-mail:  hyqq@hunnu.edu.cn}\\
Licheng Zhang \thanks{Corresponding author.} \\
{\footnotesize School of Mathematics, Hunan University, Changsha 410082, P.R.China} \\
{\footnotesize E-mail:  lczhangmath@163.com }\\
}

\date{}

\maketitle

\begin{abstract}
A drawing of a graph in the plane is called 1-planar if each edge is crossed at most once.  A graph together with a
1-planar drawing is  a 1-plane graph. A 1-plane
graph $G$ with exactly $4|V (G)|-8$ edges is called optimal. 
The crossing number $\CR(G)$ of a graph $G$ is the minimum number of crossings over all drawings of $G$. Czap and Hud\'{a}k proved that $\CR(G)\le |V(G)|-2$ for any 1-plane graph $G$  and equality holds if $G$ is an optimal 1-plane graph [{\it The Electronic J. Comb}., 20(2), \#P54 (2013)].  This paper aims to characterize maximal 1-plane graphs $G$ achieving the maximum crossing number $|V(G)|-2$. We first introduce a class of quasi-optimal 1-plane graphs as a generalization of optimal 1-plane graphs, and then prove that for  any  maximal 1-plane graph $G$, $\CR(G)=|V(G)|-2$ holds if and only if $G$ is a quasi-optimal 1-plane graph. Moreover, we prove that every quasi-optimal 1-plane graph is maximal 1-planar (not merely drawing-saturated). Finally, we present some applications of our main results, including a disproof of an upper bound on the crossing number of maximal 1-planar graphs with odd-degree vertices.
\end{abstract}

\noindent {\bf MSC}: 05C10, 05C62

\noindent {\bf Keywords}:
$1$-planar graph; Quasi-optimal;
Crossing number; Drawing

\section{Introduction \label{sec1}}
All graphs considered here are simple, finite, and undirected
unless otherwise stated, and all terminology not defined here
follows \cite{JAB}.
A {\em drawing} of a graph
$G=(V,E)$ is a mapping $D$ that assigns to each vertex in $V$ a
distinct point in the plane and to each edge $uv$ in $E$ a
continuous arc connecting $D(u)$ and $D(v)$.
We often make no
distinction between a graph-theoretical object (such as a vertex or
an edge) and its drawing.
All drawings considered here are
{\em good} unless otherwise specified, i.e., no edge crosses itself, no
two edges cross more than once, and no two edges incident with the
same vertex cross each other. We denote by $\CR_D(G)$ the number of crossings in the drawing $D$ of a graph $G$. The {\it
crossing number} $\CR(G)$ of a graph $G$ is defined as the minimum number of
crossings in any drawing of $G$.  Many papers are devoted to the study of the crossing number, see \cite{MS,MS1} and references therein.

A drawing $D$ of a graph is called $1$-{\it planar}
if each edge in $D$ is crossed at most once. A graph is 1-{\it planar} if it has a 1-planar drawing, and a graph together with a 1-planar drawing is called a 1-{\it plane graph}.
To avoid confusion, in this paper, we use
$\CR_{\times}(G)$ to denote the number of crossings in the corresponding 1-planar drawing of
the 1-plane graph $G$.

The notion of $1$-planarity was introduced in 1965 by
Ringel \cite{GR}. It is known that any 1-planar graph with $n$ vertices has at most $4n-8$ edges, and this bound is tight for $n=8$ and $n\ge 10$~\cite{IF,JP,YS1}.
A 1-planar graph is {\it maximal} if
adding any edge to it yields a graph that is either not 
1-planar or not simple.
A 1-plane graph  is
{\em maximal} if no further edge can be added to it without
violating 1-planarity or simplicity.
Clearly, the underlying graph of a maximal 1-plane graph is not necessarily maximal 1-planar.
In fact, a graph $G$ is maximal 1-planar if and only if
every 1-planar drawing of $G$ is maximal. A maximal 1-planar graph with $n$ vertices and $4n-8$ edges
is called {\it optimal}. The smallest optimal 1-planar graph is the complete 4-partite graph $K_{2,2,2,2}$, and any optimal 1-planar graph  admits a unique 1-planar
drawing (up to isomorphism) \cite{FJ1,YS}.  An optimal 1-planar graph together with its unique
1-planar drawing is called an {\it  optimal $1$-plane graph}.  

Unlike the test for planarity, the test for 1-planarity of a given graph is an NP-complete problem~\cite{VPB}. Many times, it is difficult to find a 1-planar drawing of a graph, but this does not imply that the graph is non-1-planar.
To prove a given graph is non-1-planar, several complementary approaches can be applied. The most direct is to show it exceeds the maximum edge number $4n-8$ (where $n$ is the number of vertices). Alternatively, check its chromatic number—since all 1-planar graphs are 6-colorable \cite{OVB}. Another strategy is identifying forbidden substructures; for example, using the characterization of complete multipartite 1-planar graphs in \cite{CDH}, one can detect non-1-planar subgraphs  as exclusion criteria. Finally, a technical method is to prove that any drawing of the graph contains an excessive number of crossings (at most $n-2$ for $n$ vertices; see~\cite{JCD,YS}).  For instance, with the essential help of the upper bound on the number of crossings,  Czap and Hud\'{a}k  \cite{JCD}  characterized the 1-planarity of the class of Cartesian products $K_m\Box P_n$.  

In many cases, using the crossing number of a graph to determine its 1-planarity is very effective; 
however, results on the crossing numbers of 1-planar graphs remain limited.
For a 1-planar graph $G$, Czap and Hud\'{a}k  \cite{JCD} showed  that $\CR(G)\le |V(G)|-2$.
For a maximal 1-planar graph $G$,  the first two authors of this paper and Dong \cite{OHD} showed  that $\CR(G)\le |V(G)|-2-(2\lambda_1+2\lambda_2+\lambda_3)/6$,  where for $i=1,2$,
$\lambda_i$ denotes the number of $2i$-degree vertices of $G$,
and $\lambda_3$ is the number of odd-degree vertices $w$ in $G$
such that either $d_G(w)\le 9$ or $G-w$ is $2$-connected.  For an optimal 1-planar graph $G$,  it holds that  $\CR(G)=|V(G)|-2$~\cite{OGC}.
This naturally raises the following problem:

{\it What is the structure of a 1-plane graph $G$ when $\CR(G)=|V(G)|-2$?}

In this paper, we introduce  a class of  quasi-optimal 1-plane graphs as a generalization of optimal 1-plane graphs (see Section~\ref{sec2}). In Section~\ref{sec3}, we prove our main result:

\begin{theo}\label{themain}
Let $G$ be a maximal 1-plane graph with at least $3$ vertices. Then $G$ is quasi-optimal if and only if $\CR_\times(G)=|V(G)|-2$.
\end{theo}

We also prove that Theorem~\ref{themain} still holds when $\CR_\times(G)$ is replaced by $\CR(G)$. Furthermore, we give the following theorem, which states that every quasi-optimal 1-plane graph is not only maximal 1-plane (edge-saturated with respect to the corresponding 1-planar drawing), but its underlying graph is also maximal 1-planar.
By definition, it is worth noting that maximal 1-plane graphs  and maximal 1-planar graphs are distinct concepts.
While proving that a given 1-plane graph is maximal 1-plane is often straightforward, establishing that it is maximal 1-planar is considerably more challenging, as one must consider all possible 1-planar drawings to verify maximality. Moreover, due to the lack of effective tools, to date many maximal 1-plane graphs have been constructed, but only a few families of graphs are known to be truly maximal 1-planar (see Problem 5 on Page 67 of \cite{YS}, for example).

\begin{theo}\label{themain2}
Let $G$ be a quasi-optimal 1-plane graph. Then the underlying graph of $G$ is a maximal 1-planar graph.
\end{theo}

The following problem concerning crossing numbers of maximal $1$-planar graphs was proposed in~\cite{OHD}.

\begin{prob}[\cite{OHD}]\label{prob1}
For any maximal $1$-planar graph $G$ with $n$ vertices,
does $\CR(G)\le n-2-(2\lambda_1+2\lambda_2+\lambda_3)/6$,
where, for $i=1,2$,
$\lambda_i$ denotes the number of $2i$-degree vertices of $G$,
and $\lambda_3$ is the number of odd-degree vertices in $G$?
\end{prob}

As a application of the above  two theorems,  we answer Problem~\ref{prob1} negatively by constructing infinitely many quasi-optimal 1-plane graphs with odd-degree vertices (note that any such graph $G$ is maximal 1-planar by Theorem~\ref{themain2} and satisfies $\CR(G) = |V(G)| - 2$).

\begin{theo}\label{themain3}
For every $n=14$ or $n\ge 16$, there exists  a quasi-optimal 1-plane graph with $n$ vertices containing  two odd-degree vertices.
\end{theo}

Finally, in Section~\ref{sec4}, as other applications of Theorem~\ref{themain}, we present some sufficient conditions for $\CR(G)<|V(G)|-2$.

\section{Preliminaries \label{sec2}}

The {\it connectivity} $\kappa(G)$ of a connected graph $G$ is the minimum number of vertices whose removal results in a disconnected or trivial graph. For every maximal 1-plane graph $G$,  $2\le\kappa(G)\le 7$  (see \cite{IF,YS}). A planar drawing partitions the plane into
connected regions called faces, each bounded by a closed walk (not necessarily a
cycle) termed its {\it boundary}. Two faces are {\it adjacent} if their boundaries share at least one edge. By $\partial(F)$ we denote the set of vertices on the boundary of face $F$. A face $F$ is called a {\it triangle} if $|\partial(F)|=3$, and a
{\it triangulation} (also known as maximal plane graph) is a plane graph where all faces are triangles.  

For a 1-plane graph $G$, define $G^\times$ as the plane graph obtained by replacing each crossing with  a vertex of degree 4. Vertices in  $G^\times$ are {\it fake} if they correspond to crossings in $G$, and {\it  true} otherwise.  Faces of $G^\times$ are {\it fake} if incident with a fake vertex, and {\it true} otherwise. Let $\epsilon(F)$ denote the number of true vertices on the
boundary of face $F$ of $G^\times$.  An edge $e$ in $G$ is {\it non-crossing} if it does not cross other edges,  and {\it crossing} otherwise.

Let $A$ and $B$ be two disjoint edge subsets of a graph $G$.
In a drawing $D$ of $G$, we denote by $\CR_D(A, B)$ the number of crossings between edges of $A$ and edges of $B$, and by $\CR_D(A)$ the number of crossings among edges of $A$.

\begin{defi}[Edge merging graph]\label{merging} 
Let $G_1$ and $G_2$ be 1-plane graphs with non-crossing edges $e_1$ and $e_2$, respectively. The edge merging graph $G_1 \ominus_{\{e_1,e_2\}} G_2$, abbreviated as $G_1 \ominus G_2$,  is obtained by the following steps (refer to Figure~\ref{fig00} for clarification):

\begin{itemize}[label={}, leftmargin=*, align=left]
  \item[Step 1:] Ensure $e_2$ lies on the boundary of the infinite face of $G_2^\times$ (using stereographic projection if necessary);
  \item[Step 2:] Select any face $F$ of $G_1^\times$ whose boundary contains $e_1$;
  \item[Step 3:]Insert $G_2$ into $F$ and merge $e_1$ and $e_2$ into a single edge. 
\end{itemize}

\end{defi}

\begin{figure}[h]
\centering
\includegraphics[width=0.9\textwidth]{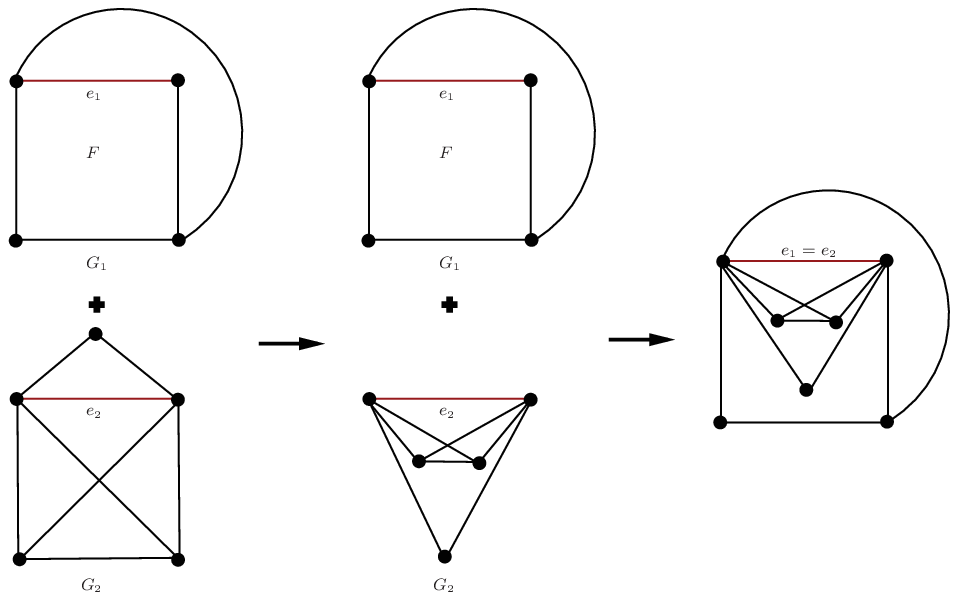}

\caption{\label{fig00} An example of edge merging operation. }
\end{figure}

 \begin{defi}[Quasi-optimal 1-plane graph]\label{mergingop}
A  quasi-optimal 1-plane graph is defined inductively as follows:
\begin{enumerate}
\item[(i)]  An optimal 1-plane graph is quasi-optimal;
\item[(ii)] The edge merging of two optimal 1-plane graphs is quasi-optimal;
\item[(iii)] The edge merging of two quasi-optimal 1-plane graphs is quasi-optimal;
\item[(iv)] Quasi-optimal 1-plane graphs are precisely those obtainable via finitely many applications of (i)–(iii).
\end{enumerate}
\end{defi}

By Definitions \ref{merging} and \ref{mergingop},  it is not difficult to verify that a quasi-optimal 1-plane graph $G$ can be decomposed into multiple optimal 1-plane graphs $G_1, \dots, G_k$ ($k \geq 1$) with $|E(G_{i})\cap E(G_{j})|\le 1$ for $i\neq j$.  We refer to $\{G_i\}_{1 \leq i \leq k}$ as the {\it generating sequence} of $G$, with each $G_i$ being a {\it generating subgraph} of $G$. Representing each $G_i$ as a vertex $v_i$ and connecting vertices $v_i$ and $v_j$ with an edge whenever $|E(G_{i})\cap E(G_{j})|=1$ yields the {\it associated graph} $G^*$ of $G$. 

\begin{pro}\label{ProA}
Let $G$ be a quasi-optimal 1-plane graph. Then,
\begin{enumerate} 
\item[(i)] The associated graph $G^*$ of $G$ is a tree;  
\item[(ii)] Either $G$ is an optimal 1-plane graph, or $G= G'\ominus H$,  where $G'$ is a quasi-optimal 1-plane graph  and $H$ is an optimal 1-plane graph.
\end{enumerate}
 \end{pro}

\begin{proof} 
(i) We proceed by induction on $n=|V(G^*)|$. The case $n=1$ is trivial. Assume  that $n\ge 2$. Then $G$ is not an optimal 1-plane graph. By Definition~\ref{mergingop}, $G$ can be expressed as $G_1\ominus G_2$, where $G_1$ and $G_2$ are quasi-optimal 1-plane graphs. This implies that $G^*=G^*_1\cup G^*_2\cup \{e\}$, where $G^*_i$ is the associated graph of $G_i$ and $e$ is the edge connecting $G^*_1$ and $G^*_2$.   By the induction hypothesis,  $G^*_1$ and $G^*_2$  are both trees.  Therefore $G^*$ is a tree, completing the induction.

(ii) Let $\{G_i\}_{1\le i\le k}$ be the generating sequence of $G$. If $k=1$, then $G$ is an optimal 1-plane graph. Assume now that $k\ge 2$. By (i), the associated graph $G^*$ of $G$ is a tree. Let $v$ be a leaf of $G^*$, which without loss of generality corresponds to $G_k$. Then $G$ can be expressed as $G' \ominus G_k$, where $G' =\bigcup_{i=1}^{k-1} G_i$ is a quasi-optimal 1-plane graph and $G_k$ is an optimal 1-plane graph. 
\end{proof}

The following conclusion follows directly from the definition of the quasi-optimal 1-plane graphs and the generating sequence.

\begin{pro}\label{obs1}
Let $G$ be a quasi-optimal 1-plane graph with  the  generating sequence $\{G_i\}_{1\le i\le k}$.  Then,
\begin{enumerate} 
\item[(i)] $|V(G)|=\sum_{i=1}^k |V(G_i)|-2(k-1)$ and $|E(G)|=\sum_{i=1}^k |E(G_i)|-(k-1)$;
\item[(ii)] $\CR_\times(G)=\sum_{i=1}^{k}\CR_\times(G_i)=|V(G)|-2$;
\item[(iii)] $G$ is 2-connected.
\end{enumerate}
\end{pro}

\begin{pro}\label{Pro1}
Let $G$ be a quasi-optimal 1-plane graph. Then
$$\frac{23}{6}|V(G)|-\frac{20}{3}\le |E(G)|\le 4|V(G)|-8.$$
Moreover, equality  holds in the upper bound if and only if  $G$ is an optimal 1-plane graph, and in the lower bound if and only if every generating subgraph  of $G$ is the complete 4-partite graph $K_{2,2,2,2}$.
\end{pro}
\begin{proof}
Let $\{G_i\}_{1\le i\le k}$ be the generating sequence of $G$.  Since each $G_i$ is an optimal 1-plane graph, $|E(G_i)|=4|V(G_i)|-8$. By Proposition~\ref{obs1} (ii), it follows that 
\begin{equation}\label{eq1}
 |E(G)|=4|V(G)|-k-7.
 \end{equation}
Note that  $k\ge 1$. The upper bound follows immediately from (\ref{eq1}), with equality  if and only if $k=1$ (i.e., $G$ is an optimal 1-plane graph).

Recall that  each $G_i$ is an optimal 1-plane graph with $|V(G_i)|\ge 8$. By Proposition~\ref{obs1} (ii), we have 
\begin{equation}\label{eq2}
|V(G)|=\sum_{i=1}^k |V(G_i)|-2(k-1)\ge 8k-2(k-1)=6k+2,
\end{equation}
which implies that
\begin{equation}\label{eq3}
k\le \frac{|V(G)|-2}{6}.
\end{equation}
Combining (\ref{eq1}) and (\ref{eq3}) yields the lower bound, and equality holds if and only if equality holds in (\ref{eq2}), i.e., $|V(G_i)| = 8$ for all $i$.  Note that the complete 4-partite graph $K_{2,2,2,2}$ is the unique optimal 1-planar graph with $8$ vertices \cite{FJ1,YS}. Thus, this
completes the proof.
\end{proof}

\begin{pro}\label{pro-add}
Let $G$ be a quasi-optimal 1-plane graph. Then $\CR(G)=|V(G)|-2$.
\end{pro}
\begin{proof}
Note that $\CR(G)\le |V(G)|-2$ holds  for any 1-plane graph $G$ \cite{JCD}. Therefore, it suffices to show that $\CR(G)\ge |V(G)|-2$.

We proceed by induction on $|V(G)|$. The conclusion is trivial if $G$ is an optimal 1-plane graph. Suppose now that $G$ is a quasi-optimal 1-plane graph but not an optimal 1-plane graph. By Proposition~\ref{ProA} (ii),  $G = G' \ominus H$ for a quasi-optimal 1-plane graph $G'$ and an optimal 1-plane graph $H$, where $|E(G') \cap E(H)|=1$ and $|V(G') \cap V(H)|=2$.   
Let $D$ be a drawing of $G$ such that $\CR(G)=\CR_D(G)$, and let $E(G') \cap E(H)=\{e\}$. Since no edge crosses itself in any good drawing, we have $\CR_D(E(G') \cap E(H))=0$. Therefore, 
\begin{align*}
\CR_D(G) &= \CR_D(G' \ominus H) \\
&= \CR_D(G') + \CR_D(H) - \CR_D(E(G') \cap E(H)) + \CR_D(E(G') \setminus e, E(H) \setminus e) \\
&\ge \CR_D(G') + \CR_D(H) \\
&\ge \CR(G') + \CR(H).
\end{align*}

As $|V(G')|<|V(G)|$,  by induction, $ \CR(G')\ge |V(G')|-2$. Recall that $\CR(H)=|V(H)|-2$ because $H$ is optimal. Thus,  noting that $|V(G') \cap V(H)|=2$, we have
\begin{align*}
\CR(G) &\ge \CR(G') + \CR(H) \\ 
&\ge |V(G')| - 2 + |V(H)| - 2 \\
&= |V(G)| - 2,
\end{align*}
as desired.
\end{proof}

\begin{lem}[\cite{OHD}]\label{lem-c}
Let $G$ be a $1$-plane graph. Then, 
$$
\CR_\times(G)=|V(G)|-2-\frac{1}{2}\sum_{F}\left(\epsilon(F)-2\right),
$$
where the sum runs over all faces $F$ of $G^{\times}$.
\end{lem}

A face $F$ of $G^\times$ is an {\it odd-face} if  $\epsilon(F)$ is odd,  and {\it even-face} otherwise. Note that $\CR_\times(G)$ is an integer. From Lemma~\ref{lem-c}, the following lemma is straightforward.
\begin{lem}\label{odd}
Let $G$ be a $1$-plane graph. Then the number of odd-face in $G^{\times}$ is even.
\end{lem}

\begin{lem}[\cite{FJ,JB}]\label{lem-adj}
Let $G$ be a maximal 1-plane graph. For any face $F$ of $G^\times$, the boundary of $F$ contains at least two true vertices; and
any two true vertices on the boundary of $F$ are adjacent in $G$.
\end{lem}

\begin{lem}\label{lem-a}
Let $G$ be a maximal 1-plane graph. If $\CR_\times(G)=|V(G)|-2$, then every face $F$ of $G^\times$ satisfies $3\le |\partial(F)|\le 4$ and $\epsilon(F)=2$.
\end{lem}
\begin{proof}
Since $G$ is a simple graph (and consequently $G^\times$ is also simple), each face $F$ of $G^\times$ must have $|\partial(F)|\ge 3$.

From Lemma~\ref{lem-c} and the assumption $\CR_\times(G)=|V(G)|-2$, we conclude that $\epsilon(F)=2$ for every face $F$ of $G^\times$. Furthermore, because no two false vertices in $G^\times$ are adjacent, the boundary of each face $F$ can contain at most four vertices, establishing $|\partial(F)|\le 4$.
\end{proof}

\section{Proofs of main  Theorems\label{sec3}}

\begin{lem}\label{lem-3con}
Let $G$ be a  maximal $1$-plane graph.
If $\kappa(G)\ge 3$ and $\CR_\times(G)=|V(G)|-2$, then $G$ is an optimal 1-plane graph.
\end{lem}

\begin{proof} 
By Lemma~\ref{lem-a}, each face of $G^\times$ has at most four vertices on its boundary, exactly two of which are true vertices.  Next, we further prove that each face of $G^\times$ is a triangle. 

\begin{cla}\label{cl1}
$G^\times$ is a triangulation.
\end{cla}
\begin{proof}
Suppose, to the contrary, that $G^\times$ has a face $\alpha$ bounded by a cycle $C$ with two true vertices $u,v$ and two false vertices $c_1,c_2$,  as shown in Figure~\ref{fig1} (a). Clearly,  $u$ and $v$ are not adjacent in $C$. By Lemma~\ref{lem-adj}, $u$ and $v$ must be adjacent in $G$.

Let $e$ be the edge in $G$ joining $u$ and $v$. We claim that $e$ is a crossing edge of $G$. Because otherwise,  $G-\{u,v\}$ is disconnected,  contradicting the fact that $\kappa(G)\ge 3$.

Assume that $e$ is crossed by $e'$ in $G$. Let $e'=xy$ and $G'=G-e'$. Let $C'$ denote the cycle formed by the edges (or segments) $uv$, $vc_1$ and $c_1u$.

Let $G_1$ and $G_2$ denote the 1-plane graphs obtained from $G'$ by removing all vertices in the exterior and interior regions of $C'$, respectively, as shown in Figure~\ref{fig1} (b) and (c). For $i=1,2$, it is not difficult to observe that there is only one face in $G_i$ whose boundary contains exactly three true vertices, i.e, $u,v$ and $x$ (or $y$). 
That is to say,  there are only one odd-face in $G_i$. By Lemma~\ref{odd}, this is impossible. Thus, the claim holds.
\end{proof}

By Claim~\ref{cl1},  $G^\times$ is a triangulation. This implies that we can obtain a maximal planar graph $G'$ from $G$ by removing one edge from each crossing pair in $G$. Therefore, 
\begin{eqnarray*}
|E(G)|&=&|E(G')|+\CR(G)\\
&=&  3|V(G')|-6+\CR(G)\\
&=&  3|V(G)|-6+|V(G)|-2\\
 &=& 4|V(G)|-8.
  \end{eqnarray*}
Thus, $G$ is an optimal 1-plane graph.
\end{proof}
\begin{figure}[h]
\centering
\includegraphics[width=1.0\textwidth]{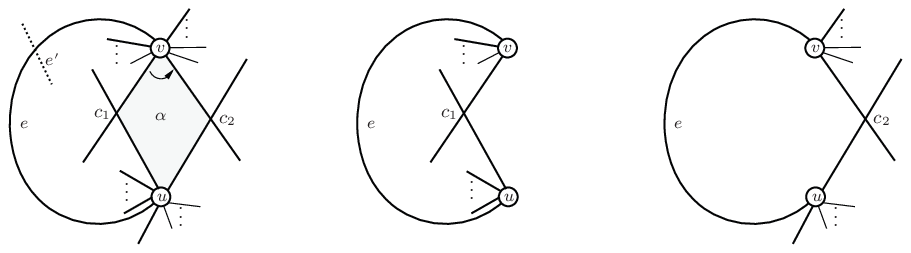}

(a) $G$ \hspace{4 cm} (b) $G_1$ \hspace{4 cm} (c) $G_2$

\caption{\label{fig1} Auxiliary graphs for proving Lemmas~\ref{lem-3con} and \ref{lem-2con}}
\end{figure}

\begin{lem}\label{lem-2con}
Let $G$ be a maximal $1$-plane graph. 
If  $\kappa(G)=2$ and $\CR_\times(G)=|V(G)|-2$, then $G$ is a quasi-optimal 1-plane graph.
\end{lem}
\begin{proof}
Suppose, to the contrary, that $G$ is a counterexample of minimum order.
Since $\CR_\times(G)=|V(G)|-2$, Lemma~\ref{lem-a} implies that each face in $G^\times$ has at most four vertices on its boundary, exactly two of which are true vertices.  If all faces of $G^\times$ are triangles, then $G$ contains a triangulation as a spanning subgraph.  Note that  any triangulation with at least 4 vertices is 3-connected~\cite{HW}.
This implies that $\kappa(G)\geq 3$, contradicting the assumption that $\kappa(G)=2$. 
Thus,  there exists  a face $\alpha$ of $G^\times$ bounded by a cycle $C$ with exactly four vertices. 

Let $u$ and $v$ be two true vertices in $C$, and let $c_1$ and $c_2$ be two false vertices in $C$, as shown in Figure~\ref{fig1} (a). Clearly,  $u$ and $v$ are not adjacent in $C$. By Lemma~\ref{lem-adj}, $u$ and $v$ must be adjacent in $G$; let $e=uv$ denote this edge. 

\begin{cla}\label{cl2}
 $e$ is a non-crossing edge in $G$.
\end{cla}
\begin{proof}
Assume that $e$ is a crossing edge in $G$. A contradiction then follows by arguments analogous to those in the proof of Claim~\ref{cl1}  in Lemma~\ref{lem-3con}; we omit the repetitive details.
\end{proof}

Let $G_1$ and $G_2$ denote the 1-plane graphs obtained from $G$ by removing all vertices in the exterior and interior regions of $C$, respectively, as shown in Figure~\ref{fig1} (b) and (c).  Observe that each $G_i$ $(i=1,2)$ remains maximal 1-plane graph and  in $G_i^\times$,  the boundary of every face  contains exactly two true vertices.  Thus, for $i=1,2$, it follows that $\CR_\times(G_i)=|V(G_i)|-2$ from Lemma~\ref{lem-c}.  

For $i\in\{1,2\}$, if $\kappa(G_i)=2$, then  $G_i$ is a quasi-optimal 1-plane graph by the minimality assumption on $|V(G)|$; if $\kappa(G_i)\ge 3$,  then $G_i$ is an optimal 1-plane graph by Lemma~\ref{lem-3con}. Consequently, Definition~\ref{mergingop} implies that $G$ is a quasi-optimal 1-plane graph, contradicting our initial assumption. 

This completes the proof.
\end{proof}

Combining Lemmas~\ref{lem-3con} and~\ref{lem-2con}, we can immediately complete the proof of Theorem~\ref{themain}.\\

\noindent{\bf Proof of Theorem~\ref{themain}:}
If $G$ is a quasi-optimal 1-plane graph, Proposition~\ref{obs1} (ii) gives $\CR_\times(G)=|V(G)|-2$. The converse follows from Lemmas~\ref{lem-3con} and~\ref{lem-2con}. \hfill 
$\square$

\begin{lem}\label{lem6}
Let $G$ be a maximal 1-plane graph with at least 3 vertices. Then,  $\CR_\times(G)=|V(G)|-2$  if and only if  $\CR(G)=|V(G)|-2$.
\end{lem}
\begin{proof}
Assume that $\CR_\times(G)=|V(G)|-2$.  Since $\kappa(G) \geq 2$ for any maximal $1$-plane graph with at least $3$ vertices, Lemmas~\ref{lem-3con} and~\ref{lem-2con} imply that $G$ is a quasi-optimal 1-plane graph. Thus Proposition~\ref{pro-add} yields $\CR(G) = |V(G)| - 2$.

Assume that $\CR(G)=|V(G)|-2$.  Then $|V(G)|-2=\CR(G)\le \CR_\times(G)\le |V(G)|-2$, implying  $\CR_\times(G)=|V(G)|-2$
\end{proof}

\begin{rem}\label{remark1}
By Lemma~\ref{lem6}, Theorem~\ref{themain} remains valid when $\CR_\times(G) = |V(G)| - 2$ is replaced by $\CR(G)=|V(G)|-2$.
\end{rem}

\noindent{\bf Proof of Theorem~\ref{themain2}:}
We proceed by induction on $|V(G)|$. If $G$ is an optimal 1-plane graph, the conclusion is trivial. Assume that $G$ is a quasi-optimal 1-plane graph but not an optimal 1-plane graph. By Proposition~\ref{ProA} (ii),  $G = G' \ominus H$ for a quasi-optimal 1-plane graph $G'$ and an optimal 1-plane graph $H$, with   $|E(G') \cap E(H)|=1$. Clearly, $|V(G')|<|V(G)|$ and $|V(H)|<|V(G)|$.  

We denote the underlying graphs of $G$, $G' $ and $H$ by the same symbols.   By \cite{YS1}, $H$ admits a unique 1-planar drawing $\varphi$ where every face is a triangular false face. Let $D$ be an arbitrary 1-planar drawing of $G$; note that $D$ must contain $\varphi$.  

Let $e = uv$ be the unique edge in $E(G') \cap E(H)$. Suppose that $e$ is a crossing edge in $\varphi$. Then vertices in $V(G') \setminus \{u, v\}$  must lie in the regions of $\varphi$ incident with $v$ and $u$ (see Figure~\ref{fig0}, left). Since the vertices in the regions incident with $v$ and the regions incident with $u$ cannot be connected by edges without violating 1-planarity, $G'-e$ is disconnected, contradicting the 2-connectivity of $G'$. Thus, $e$ must be non-crossing in $\varphi$.    

Moreover, all vertices in $V(G') \setminus \{u, v\}$ must lie in the shaded region bounded by four crossing segments $u\alpha_1,v\alpha_1,u\alpha_2,v\alpha_2$ in $\varphi$  (see Figure~\ref{fig0}, right). Deviating from this would violate 1-planarity or 2-connectivity of $D$. This configuration prohibits adding edges between $V(G')$ and $V(H)$. By induction,  $G'$ and $H$ are both maximal 1-planar graph.  Certainly, no edges can be added within $V(G')$ or $V(H)$ due to the maximality of $G'$ and $H$. Since $D$ was arbitrary, $G$ is maximal.  This completes the induction. 
\hfill $\square$

\begin{figure}[h]
\centering
\includegraphics[width=0.8\textwidth]{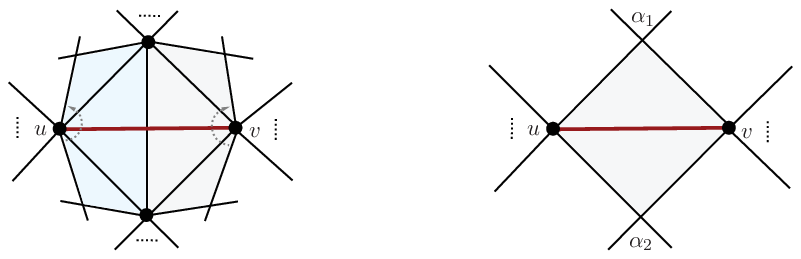}

\caption{\label{fig0} The local  configuration of the drawing $D(H)$. }
\end{figure}

\noindent{\bf Proof of Theorem~\ref{themain3}:}
Let $G_1$ and $G_2$ be optimal 1-plane graphs with $n_1$ and $n_2$ vertices, respectively. By~\cite{JP,YS1},  $n_1,n_2 \in \{8\} \cup \{k\in\mathbb{N} \mid k\geq 10\}$. Perform an edge merging operation (Definition~\ref{merging}) on $G_1$ and $G_2$ to obtain a quasi-optimal 1-plane graph $G$ with $|V(G)|= n_1 +n_2-2$. Thus $|V(G)| \in \{14\} \cup \{k \in \mathbb{N} \mid k \geq 16\}$. 

Since optimal 1-plane graphs are Eulerian~\cite{YS1}, the edge merging operation introduces exactly two odd-degree vertices in $G$, establishing the result.
\hfill $\square$

\section{Applications \label{sec4}}

If a 1-plane graph $G$ can be extended to a quasi-optimal 1-plane graph by adding edges, then $G$ is called an {\it EQ-graph}. The following conclusion is straightforward by Theorem~\ref{themain}.

\begin{cor}\label{cro0}
Let $G$ be a 1-plane graph with $n$ vertices.  If $G$ is not an EQ-graph, then $G$ has at most $n-3$ crossings.
\end{cor}

Recall that any quasi-optimal 1-plane graph has at least 8 vertices, and the same is true for EQ-graphs.  
Thus, from Corollary~\ref{cro0} the following conclusion is obvious.

\begin{cor}\label{cro1}
Every 1-plane graph with $3\le n\le 7$ vertices has at most $n-3$ crossings.
\end{cor}

Let $P_n$ and $C_n$ be the path and the cycle on $n$ vertices, respectively.  Let $H$ be a subgraph of graph $G$. We denote by $G-H$ the graph obtained by removing all edges  of $H$ from $G$.
A graph $G$ is a {\it minimal} non-1-planar graph (MN-graph, for short) if $G$ is non-1-planar, but $G-e$ is 1-planar for every edge $e$ of $G$. In \cite{VPK}, Korzhik  proved that the graph $K_7-K_3$ is the unique MN-graph with 7 vertices.  Their proof also implies the following conclusion. Here, we provide a new proof using the ``crossing number".

\begin{cor}\label{croa}
A graph with at most 7 vertices is non-1-planar if and only if it is one of $K_7$, $K_7-P_2$, $K_7-P_3$, or $K_7-C_3$.
\end{cor}
\begin{proof} 
All graphs with fewer than 7 vertices are 1-planar since even $K_6$ is 1-planar \cite{YS1}. Thus we only consider graphs with 7 vertices. From \cite{MKS} it follows that $\CR(K_7-C_3)=\CR(K_{1,1,1,1,3})=5>7-3=4$. Hence, by Corollary~\ref{cro1}, $K_7-C_3$ is non-1-planar.  Since each of $K_7$, $K_7-P_2$, and $K_7-P_3$  contains $K_7-C_3$ as a subgraph, they are also non-1-planar. 

It suffices to show that both $K_7-2K_2$  and $K_7-K_{1,3}$ are 1-planar, where $2K_2$ denotes two independent edges. This is because all graphs with 7 vertices except $K_7$, $K_7-P_2$, $K_7-P_3$, and $K_7-C_3$ are subgraphs of $K_7-2K_2$  or $K_7-K_{1,3}$.  The 1-planar drawings shown in Figure~\ref{fig2} confirm the 1-planarity of $K_7-2K_2$  and $K_7-K_{1,3}$, thus completing the proof.
\end{proof}

\begin{figure}[h]
\centering
\includegraphics[width=0.7\textwidth]{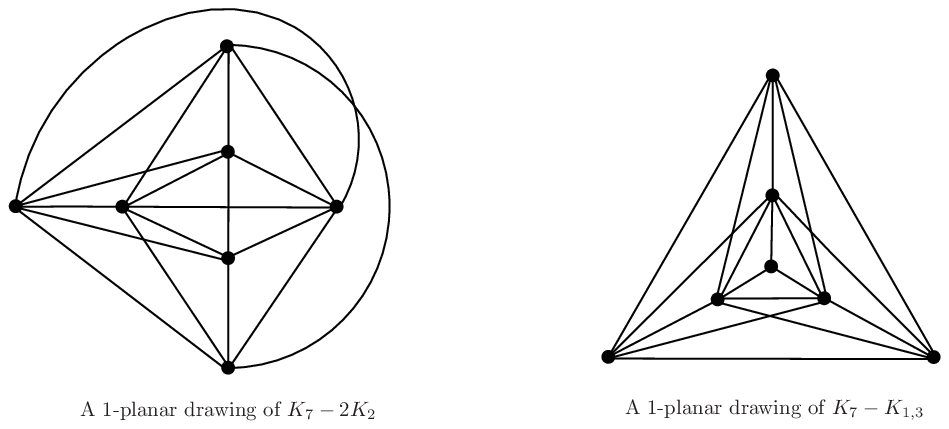}
\caption{\label{fig2} The 1-planar drawings of $K_7-2K_2$  and $K_7-K_{1,3}$}
\end{figure}

A vertex of a graph $G$ is a {\it dominating vertex} if it is adjacent to all other vertices in $G$. As optimal $1$-plane graphs contain no dominating vertices~\cite{OGC}, neither do quasi-optimal $1$-plane graphs. This means that no 1-plane graphs with  dominating vertices are EQ-graphs.
Thus, we deduce the following conclusion from Corollary~\ref{cro0}.

\begin{cor}\label{cro2}
Every 1-plane graph with $n\ge 3$ vertices and a dominating vertex has at most $n-3$ crossings.
\end{cor}

By Proposition~\ref{Pro1},  any optimal 1-plane graph with $n$ vertices has at least $\frac{23}{6}n-\frac{20}{3}$ edges. Thus, the following corollary is immediate.

\begin{cor}\label{cro3}
Every maximal 1-plane graph with $n\ge 3$ vertices and $m< \frac{23}{6}n-\frac{20}{3}$ edges has at most $n-3$ crossings.
\end{cor}

\end{document}